\documentclass[letterpaper]{amsart}

\newtheorem{theorem}{Theorem}[section]

\newtheorem{proposition}[theorem]{Proposition}
\newtheorem{corollary}[theorem]{Corollary}

\theoremstyle{definition}
\newtheorem{definition}[theorem]{Definition}

\theoremstyle{remark}
\newtheorem{remark}[theorem]{Remark}

\usepackage[pdftex]{graphicx}
\usepackage{color}
\usepackage{amsmath}
\usepackage{amsfonts}
\usepackage{amssymb}
\usepackage{epsfig}
\usepackage{rotating}
\usepackage{graphics}
\newcommand{\be}{\begin{equation}}

\newcommand{\ee}{\end{equation}}

\newcommand{\g}{\gamma}

\newcommand{\om}{\omega}

\newcommand{\weyl}{{\stackrel{\scriptscriptstyle{LC}}{W}}\phantom{}}
\newcommand{\schou}{{\stackrel{\scriptscriptstyle{LC}}{\Rho}}}
\newcommand{\rlc}{{\stackrel{\scriptscriptstyle{LC}}{R}}}
\newcommand{\oms}{{\stackrel{\scriptscriptstyle{LC}}{\omega}}}

\newcommand{\D}{{\rm D}}

\newcommand{\dz}{\wedge}

\newcommand{\ba}{\begin{array}}

\newcommand{\ea}{\end{array}}

\newcommand{\beq}{\begin{eqnarray}}

\newcommand{\eeq}{\end{eqnarray}}

\newtheorem{lm}{lemma}

\newtheorem{thee}{theorem}

\newtheorem{proo}{proposition}

\newtheorem{co}{corollary}

\newtheorem{rem}{remark}

\newtheorem{deff}{definition}

\newcommand{\bd}{\begin{deff}}

\newcommand{\ed}{\end{deff}}

\newcommand{\bl}{\begin{lm}}

\newcommand{\el}{\end{lm}}

\newcommand{\bp}{\begin{proo}}

\newcommand{\ep}{\end{proo}}

\newcommand{\bt}{\begin{thee}}

\newcommand{\et}{\end{thee}}

\newcommand{\bc}{\begin{co}}

\newcommand{\ec}{\end{co}}

\newcommand{\brm}{\begin{rem}}

\newcommand{\erm}{\end{rem}}

\newcommand{\der}{{\rm d}}

\usepackage{t1enc}
\def\frak{\mathfrak}

\newcommand{\newc}{\newcommand}

\let\ccdot\cdot
\def\cdot{\hbox to 2.5pt{\hss$\ccdot$\hss}}

\newc{\aR}{\mbox{\boldmath{$ R$}}}
\newc{\aS}{\mbox{\boldmath{$ S$}}}
\newc{\aT}{\mbox{\boldmath{$ T$}}}
\newc{\aW}{\mbox{\boldmath{$ W$}}}

\newc{\aK}{\mbox{\boldmath{$ K$}}}
\newc{\aL}{\mbox{\boldmath{$ L$}}}

\usepackage{amssymb}
\usepackage{amscd}

\newcommand{\Rho}{{\mbox{\sf P}}}

\newcommand{\hook}{\raisebox{-0.35ex}{\makebox[0.6em][r]
{\scriptsize $-$}}\hspace{-0.15em}\raisebox{0.25ex}{\makebox[0.4em][l]{\tiny
 $|$}}}

\newcommand{\bma}{\begin{pmatrix}}
\newcommand{\ema}{\end{pmatrix}}

\newcommand{\X}{\mbox{\boldmath{$ X$}}}

\newc{\obstrn}[2]{B^{#1}_{#2}}

\newcommand{\rpl}                         
{\mbox{$
\begin{picture}(12.7,8)(-.5,-1)
\put(0,0.2){$+$}
\put(4.2,2.8){\oval(8,8)[r]}
\end{picture}$}}

\newcommand{\lpl}                         
{\mbox{$
\begin{picture}(12.7,8)(-.5,-1)
\put(2,0.2){$+$}
\put(6.2,2.8){\oval(8,8)[l]}
\end{picture}$}}

\usepackage{ifthen}

\newc{\tensor}[1]{#1}
\newc{\Mvariable}[1]{\mbox{#1}}
\newc{\down}[1]{{}_{#1}}
\newc{\up}[1]{{}^{#1}}

\newc{\JulyStrut}{\rule{0mm}{6mm}}
\newc{\midtenPan}{\mbox{\sf S}}
\newc{\midten}{\mbox{\sf T}}
\newc{\midtenEi}{\mbox{\sf U}}
\newc{\ATen}{\mbox{\sf E}}
\newc{\BTen}{\mbox{\sf F}}
\newc{\CTen}{\mbox{\sf G}}

\def\sideremark#1{\ifvmode\leavevmode\fi\vadjust{\vbox to0pt{\vss
 \hbox to 0pt{\hskip\hsize\hskip1em
 \vbox{\hsize3cm\tiny\raggedright\pretolerance10000
 \noindent #1\hfill}\hss}\vbox to8pt{\vfil}\vss}}}%

\newcommand{\Span}{\mathrm{Span}}

\newcounter{romenumi}
\newcommand{\labelromenumi}{(\roman{romenumi})}

\begin{document}
\title{Projective vs metric structures}
\vskip 1.truecm
\author{Pawe\l~ Nurowski} \address{Instytut Fizyki Teoretycznej,
Uniwersytet Warszawski, ul. Hoza 69, Warszawa, Poland}
\email{nurowski@fuw.edu.pl} 

\date{\today}

\begin{abstract}
We present a number of conditions which are necessary 
for an $n$-dimensional projective structure $(M,[\nabla])$ 
to include the Levi-Civita connection $\nabla$ of some metric on $M$. We provide
an algorithm, which effectively checks if a Levi-Civita
connection is in the projective class and, in the positive,
which finds this connection and the metric. 
The article also provides a basic information on invariants of
projective structures, including the treatment via Cartan's normal
projective connection. In particular we show that there is a number
of Fefferman-like conformal structures, defined on a subbundle of the Cartan
bundle of the projective structure, which encode the projectively
invariant information about $(M,[\nabla])$.  

\end{abstract}
\maketitle
\tableofcontents
\newcommand{\bbS}{\mathbb{S}}
\newcommand{\bbR}{\mathbb{R}}
\newcommand{\sog}{\mathbf{SO}}
\newcommand{\slg}{\mathbf{SL}}
\newcommand{\glg}{\mathbf{GL}}
\newcommand{\og}{\mathbf{O}}
\newcommand{\soa}{\frak{so}}
\newcommand{\sla}{\frak{sl}}
\newcommand{\sua}{\frak{su}}
\newcommand{\dr}{\mathrm{d}}
\newcommand{\sug}{\mathbf{SU}}
\newcommand{\gat}{\tilde{\gamma}}
\newcommand{\Gat}{\tilde{\Gamma}}
\newcommand{\thet}{\tilde{\theta}}
\newcommand{\Thet}{\tilde{T}}
\newcommand{\rt}{\tilde{r}}
\newcommand{\st}{\sqrt{3}}
\newcommand{\kat}{\tilde{\kappa}}
\newcommand{\kz}{{K^{{~}^{\hskip-3.1mm\circ}}}}
\newcommand{\bv}{{\bf v}}
\newcommand{\di}{{\rm div}}
\newcommand{\curl}{{\rm curl}}
\newcommand{\cs}{(M,{\rm T}^{1,0})}
\newcommand{\tn}{{\mathcal N}}

\section{Projective structures and their invariants}
\subsection{Definition of a projective structure}
A \emph{projective structure} on an $n$-dimensional manifold $M$ is an
equivalence class of \emph{torsionless} connections $[\nabla]$ with an
equivalence relation identifying every two connections $\hat{\nabla}$ and 
$\nabla$ for which 
\be
\hat{\nabla}_XY=\nabla_XY+A(X)Y+A(Y)X,\quad\quad\quad\forall X,Y\in{\rm T}M,\label{prst}\ee
with some 1-form $A$ on $M$. 

Two connections from a projective class have \emph{the same 
unparametrized geodesics} in $M$, and the converse is also true: two
torsionless connections have the same unparametrized geodesics in $M$
if they belong to the same projective class.

The main pourpose of this article is to answer the following question:

`When a given projective class of connections $[\nabla]$ on $M$ includes a
Levi-Civita connection of some metric $g$ on $M$?'

This problem has a long history, see e.g. \cite{liu,mik,sin}. It was
recently solved in $\dim M=2$ in a beatiful paper \cite{bde}, which
also, in its last section, indicates how to treat the problem in $\dim
M\geq 3$. In the present paper we follow \cite{bde} and treat the
problem in full generality\footnote{I have been recently informed by M Dunajski that the problem is also
being considered by him and S Casey \cite{cas}.} in $\dim M\geq 3$. On doing this we need
the \emph{invariants} of projective structures.

The system of local invariants for projective structures was
constructed by Cartan \cite{car} (see also \cite{tom}). We briefely present it here for the
completness (see e.g. \cite{ea,kob,nn} for more details).

For our pourposes it is convenient to describe a connection $\nabla$
in terms of the connection coefficients $\Gamma^i_{~jk}$ associated
with any frame $(X_a)$ on $M$. This is possible via the formula: 
$$\nabla_{a}X_b=\Gamma^c_{~ba}X_c,\quad\quad\quad \nabla_a:=\nabla_{X_a}.$$ Given a frame $(X_a)$ these
relations provide a one-to-one correspondence between connections
$\nabla$ and the connection coefficients $\Gamma^a_{~bc}$. In particular,
a connection is torsionless iff 
$$\Gamma^c_{~ab}-\Gamma^c_{~ba}=-\theta^c([X_a,X_b]),$$ where
$(\theta^a)$ is a coframe dual to $(X_a)$,
$$\theta^b(X_a)=\delta^b_{~a}.$$ 
Moreover, two connections $\hat{\nabla}$ and
$\nabla$ are in the same projective class iff there exists a coframe in which 
$$\hat{\Gamma}^c_{~ab}=\Gamma^c_{~ab}+\delta^c_{~a}A_b+\delta^c_{~b}A_a,$$
for some 1-form $A=A_a\theta^a$. 

In the following, rather than using the connection coefficients, we
will use a collective object
$$\Gamma^a_{~b}=\Gamma^a_{~bc}\theta^c,$$
which we call connection 1-forms. In terms of them the projective
equivalence reads:
\be
\hat{\Gamma}^a_{~b}=\Gamma^a_{~b}+\delta^a_{~b}A+A_b\theta^a.\label{pt}\ee

\subsection{Projective Weyl, Schouten and Cotton tensors}
Now, given a projective structure $[\nabla]$ on $M$, we take
a connection 1-forms $(\Gamma^i_{~j})$ of a particular representative
$\nabla$. Because of no torsion we have:
\be\der\theta^a+\Gamma^a_{~b}\dz\theta^b=0.\label{c11}\ee 
The curvature of this connection
\be\Omega^a_{~b}=\der\Gamma^a_{~b}+\Gamma^a_{~c}\dz\Gamma^c_{~b},\label{c12}\ee
which defines the curvature tensor $R^a_{~bcd}$ via:
$$\Omega^a_{~b}=\tfrac12 R^a_{~bcd}\theta^c\dz\theta^d,$$
is now decomposed onto the irreducible components with respect to the
action of $\glg(n,\bbR)$ group:
\be
\Omega^a_{~b}=W^a_{~b}+\theta^a\dz\omega_b+\delta^a_{~b}\theta^c\dz\omega_c.\label{c13}\ee
Here $W^a_{~b}$ is endomorphims-valued 2-form:
$$W^a_{~b}=\tfrac12W^a_{~bcd}\theta^c\dz\theta^d,$$
which is totally traceless:
$$W^a_{~a}=0,\quad\quad W^a_{~bac}=0,$$
and has all the symmetries of $R^a_{~bcd}$. 
Quantity $\omega_a$ is a covector-valued 1-form. It defines a
tensor $\Rho_{ab}$ via 
\be
\omega_b=\theta^a\Rho_{ab}.\label{d1}\ee
The tensors $W^a_{~bcd}$ and $P_{ab}$ are called the Weyl tensor, and
the Schouten tensor, respectively. They are realted to the curvature
tensor $R^a_{~bcd}$ via:
$$R^a_{~bcd}=W^a_{~bcd}+\delta^a_{~c}\Rho_{db}-\delta^a_{~d}\Rho_{cb}-
2\delta^a_{~b}\Rho_{[cd]}.$$
In particular, we have also the relation between the Schouten tensor
$\Rho_{ab}$ and the Ricci tensor
$$R_{ab}=R^c_{~acb},$$ 
which reads:
\be
\Rho_{ab}=\tfrac{1}{n-1}R_{(ab)}-\tfrac{1}{n+1}R_{[ab]}.\label{c41}\ee
One also introduces the Cotton tensor $Y_{bca}$, which is defined via 
the covector valued 2-form 
\be Y_a=\tfrac12Y_{bca}\theta^b\dz\theta^c,\label{c14}\ee
by 
\be Y_a=\der\omega_a+\omega_b\dz\Gamma^b_{~a}.\label{c15}\ee
Note that $Y_{bca}$ is antisymmetric in $\{bc\}$.

Now, combining the equations (\ref{c11}), (\ref{c12}), (\ref{c13}),
(\ref{c14}) and (\ref{c15}), we get the \emph{Cartan structure
  equations}:
\be\begin{aligned}
&\der\theta^a+\Gamma^a_{~b}\dz\theta^b=0\\
&\der\Gamma^a_{~b}+\Gamma^a_{~c}\dz\Gamma^c_{~b}=
W^a_{~b}+\theta^a\dz\omega_b+\delta^a_{~b}\theta^c\dz\omega_c\\
&\der\omega_a+\omega_b\dz\Gamma^b_{~a}=Y_a.
\end{aligned}\label{cs}\ee 

It is convenient to introduce the \emph{covariant exterior
  differential} $\D$, which on tensor-valued $k$-forms acts as:
$$\D K^{a_1\dots a_r}_{\quad\quad\,\,\,b_1\dots b_s}=\der K^{a_1\dots a_r}_{\quad\quad\,\,\,b_1\dots
  b_s}+
\Gamma^{a_i}_{\,\,\,\,a}\dz K^{a_1\dots a\dots a_r}_{\quad\quad\quad\,\,\,\, b_1\dots
  b_s}-\Gamma^b_{~b_i}\dz K^{a_1\dots a_r}_{\quad\quad\,\,\,b_1\dots
  b\dots b_s}.$$
This, in particular satisfies the Ricci identity:
\be
\D^2K^{a_1\dots a_r}_{\quad\quad\,\,\,b_1\dots b_s}=\Omega^{a_i}_{\,\,\,\,a}\dz K^{a_1\dots a\dots a_r}_{\quad\quad\quad\,\,\,\, b_1\dots
  b_s}-\Omega^b_{~b_i}\dz K^{a_1\dots a_r}_{\quad\quad\,\,\,b_1\dots
  b\dots b_s}.\label{ri}\ee
This identity will be crucial in the rest of the paper.

Using $\D$ we can write the first and the third Cartan structure
equation in respective compact forms:
\be\begin{aligned}
&\D\theta^a=0,\\
&\D\omega_a=Y_a.
\end{aligned}\label{csm}
\ee
Noting that on tensor-valued 0-forms we have: 
$$\D K^{a_1\dots a_r}_{\quad\quad\,\,\,b_1\dots b_s}=
\theta^c\nabla_cK^{a_1\dots a_r}_{\quad\quad\,\,\,b_1\dots b_s},$$
and comparing with the definition (\ref{d1}) one sees that the second equation
(\ref{csm}) is equivalent to:
\be
Y_{bca}=2\nabla_{[b}\Rho_{c]a}.\label{dd1}\ee
\subsection{Bianchi identities}
We now apply $\D$ on the both sides of the Cartan structure equations
(\ref{cs}) and use the Ricci formula (\ref{ri}) to obtain the Bianchi 
identities.

Applying $\D$ on the first of (\ref{cs}) we get
$$0=\D^2\theta^a=\Omega^a_{~b}\dz\theta^b,$$ i.e. tensorially:
$$R^a_{~[bcd]}=0.$$
This, because the Weyl tensor has the same symmetries as $R^a_{~bcd}$, means also
that
\be
W^a_{~[bcd]}=0.\label{bi1}\ee

Next, applying $\D$ on the second of (\ref{cs}) we get:
$$\D W^a_{~b}=\theta^a\dz Y_b+\delta^a_{~b}\theta^c\dz Y_c.$$
This, when written in terms of the tensors $W^a_{~bcd}$ and $Y_{abc}$,
reads:
\be
\begin{aligned}
&\nabla_a W^d_{~ebc}+\nabla_c W^d_{~eab}+\nabla_b W^d_{~eca}=\\
&\quad\quad\delta^d_{~a}Y_{bce}+\delta^d_{~c}Y_{abe}+\delta^d_{~b}Y_{cae}+
\delta^d_{~e}(Y_{abc}+Y_{cab}+Y_{bca}).
\end{aligned}\label{13}
\ee
This, when contracted in $\{ad\}$, and compared with (\ref{bi1}),
implies in particular that:
\be\nabla_dW^d_{~abc}=(n-2)Y_{bca}\label{bi2}\ee
and
\be
Y_{[abc]}=0.\label{bi3}\ee
Thus when $n>2$ the Cotton tensor is determined by the divergence of
the Weyl tensor. 

It is also worthwhile to note, that because of (\ref{bi3}) the
identity (\ref{13}) simplifies to:
\be
\begin{aligned}
\nabla_a W^d_{~ebc}+\nabla_c W^d_{~eab}+\nabla_b W^d_{~eca}=\delta^d_{~a}Y_{bce}+\delta^d_{~c}Y_{abe}+\delta^d_{~b}Y_{cae}.
\end{aligned}
\ee

Another immediate but useful consequence of the
identity (\ref{bi3}) is 
\be
\nabla_{[a}\Rho_{bc]}=0.\label{bi76}\ee
This fact suggests an introduction of a 2-form 
$$\beta=\tfrac12\Rho_{[ab]}\theta^a\dz\theta^b.$$
Since $\beta$ is a \emph{scalar} 2-form we have:
$$\begin{aligned}
&\der\beta=\D\beta=\D(\tfrac12\Rho_{[ab]}\theta^a\dz\theta^b)=\\
&\tfrac12(\D\Rho_{[ab]})\theta^a\dz\theta^b=\\
&\tfrac12(\nabla_c\Rho_{[ab]})\theta^c\dz\theta^a\dz\theta^b=\\
&\tfrac12(\nabla_{[c}\Rho_{ab]})\theta^c\dz\theta^a\dz\theta^b=0.\end{aligned}$$
Thus, due to the Bianchi identity (\ref{bi76}) and the first structure
equation (\ref{csm}), the 2-form $\beta$ is
\emph{closed}. 

Finally, applying $\D$ on the last Cartan equation (\ref{cs}) we get
$$\D Y_a+\omega_b\dz W^b_{~a}=0.$$
This relates 1st derivatives of the Cotton tensor to a bilinear
combination of the Weyl and the 
Schouten tensors:
\be\begin{aligned}
&\nabla_aY_{bcd}+\nabla_cY_{abd}+\nabla_bY_{cad}=\\
&\quad\quad\quad\Rho_{ae}W^e_{~dcb}+\Rho_{be}W^e_{~dac}+\Rho_{ce}W^e_{~dba}.
\end{aligned}\ee
\subsection{Gauge transformations} It is a matter of checking that if
we take another connection $\hat{\nabla}$ from the projective class
$[\nabla]$, i.e. if we start with connection 1-forms
$\hat{\Gamma}^i_{~j}$ related to $\Gamma^i_{~j}$ via 
$$\hat{\Gamma}^a_{~b}=\Gamma^a_{~b}+\delta^a_{~b}A+A_b\theta^a,$$
then the basic objects $\omega_a$, $W^a_{~b}$ and $Y_a$ transform as:
\be\begin{aligned}
&\hat{\omega}_a=\omega_a-\D A_a+A A_a\\
&\hat\beta=\beta-\der A\\
&\hat{W}^a_{~b}=W^a_{~b}\\
&\hat{Y}_a=Y_a+A_bW^b_{~a}.
\end{aligned}\label{1tr}
\ee
This, in the language of 0-forms means:
\be\begin{aligned}
&\hat{\Gamma}^a_{~bc}=\Gamma^a_{~bc}+\delta^a_{~c}A_b+\delta^a_{~b}A_c\\
&\hat{P}_{ab}=P_{ab}-\nabla_aA_b+A_aA_b\\
&\hat{P}_{[ab]}=P_{[ab]}-\nabla_{[a}A_{b]}\\
&\hat{W}^a_{~bcd}=W^a_{~bcd}\\
&\hat{Y}_{abc}=Y_{abc}+A_dW^d_{~cab}.
\end{aligned}\label{2tr}
\ee
This in particular means that \emph{the Weyl tensor is a projectively
  invariant object}. We also note that the 2-form $\beta$ \emph{transforms
modulo addition of a total differential}.
\begin{corollary}
Locally in every projective class $[\nabla]$ there exists a
torsionless connection $\nabla^0$ for which the Schouten tensor is
symmetric, $\Rho_{ab}=\Rho_{(ab)}$. 
\end{corollary}
\begin{proof}
We know that due to the Bainchi identities (\ref{bi76}) the 2-form
$\beta$ encoding the antisymmetric part of $\Rho_{ab}$ is closed,
$\der\beta=0$. Thus, using the Poincare lemma, we know that 
there exists a 1-form $\Upsilon$ such that locally
$\beta=\der\Upsilon$. It is therefore sufficient to take $A=\Upsilon$
and
$\hat{\Gamma}^a_{~b}=\Gamma^a_{~b}+\delta^a_{~b}\Upsilon+\theta^a\Upsilon_b$,
to get $\hat{\beta}=0$, by the second relation in (\ref{1tr}). This
proves that in the connection $\hat{\Gamma}^a_{~b}$ projectively
equivalent to $\Gamma^a_{~b}$, we have $\hat{\Rho}_{[ab]}=0$. 
\end{proof}
\begin{remark}\label{ko}
Note that if $\Gamma^{a}_{~b}$ is a connection for which $\Rho_{ab}$
is symmetric then it is also symmetric in any projectively equivalent
connection for which $A=\der\phi$, where $\phi$ is a function.
\end{remark}
\begin{definition}
A subclass of projectively equivalent connections for which the
Schouten tensor is \emph{symmetric} is called \emph{special}
projective class.
\end{definition}
Mutatis mutandis we have: 
\begin{corollary}
Locally every projective class contains a \emph{special}
projective subclass. This subclass is given modulo transformations
(\ref{pt}) with $A$ being a gradient, $A=\der\phi$.
\end{corollary}
\begin{corollary}\label{kor}
The curvature $\Omega^a_{~b}$ of any connection from a special
projective subclass of projective connections $[\nabla]$ is tarceless, $\Omega^a_{~a}=0$. 
\end{corollary}
\begin{proof}
For the connections from a special projective subclass we have
$\Rho_{ab}=\Rho_{ba}$. Hence
$\theta^a\dz \omega_a=\theta^a\dz\Rho_{ba}\theta^b\dz=0$, and
$\Omega^a_{~b}=W^a_{~b}+\theta^a\dz\omega_b$. Thus
$$\Omega^a_{~a}=W^a_{~a}+\theta^a\dz\omega_a=0,$$ because 
the Weyl form $W^a_{~b}$ is traceless.   
\end{proof}
\begin{remark}
We also remark that \emph{in dimension} $n=2$ \emph{the Weyl
tensor of a projective structure is identically zero}. In this
dimension the Cotton tensor provides the lowest order projective
invariant (see the last equation in (\ref{2tr})). In
dimension $n=3$ the Weyl tensor is generically non-zero, and may have as much as
\emph{fifteen} indpenednet components. It is also generiaclly nonzero
in dimensions higher than three.
\end{remark}

Given an open set $\mathcal U$ with coordinates $(x^a)$ surely the simplest
projective structure $[\nabla]$ is the one represented by the
connection $\nabla_a=\frac{\partial}{\partial x^a}$. This is called
the \emph{flat} projective structure on $\mathcal U$. The following
theorem is well known \cite{car,tom}:
\begin{theorem}
In dimension $n\geq 3$ a projective structure $[\nabla]$ is locally
projectively equivalent to the flat projective structure if and only
if its projective Weyl tensor vanishes identically, $W^a_{~bcd}\equiv 0$. In dimension $n=2$
a projective structure $[\nabla]$ is locally
projectively equivalent to the flat projective structure if and only
if its projective Schouten tensor vanishes identically, $Y_{abc}\equiv 0$.
\end{theorem}
\subsection{Cartan connection}\label{ser}
Objects $(\theta^a,\Gamma^b_{~c},\omega_d)$ can be collected to the
  \emph{Cartan connection} on an $H$ principal fiber bundle $H\to P\to
  M$ over
  $(M,[\nabla])$. Here $H$ is a subgroup of the $\slg(n+1,\bbR)$ group
  defined by:
$$H=\{ b\in\slg(n+1,\bbR)~|~b=\bma
  A^a_{~b}&0\\A_b&a^{-1}\ema,~A^a_{~b}\in\glg(n,\bbR),~A_a\in
  (\bbR^n)^*, ~a=\det(A^a_{~b})~\}.$$
Using $(\theta^a,\Gamma^b_{~c},\omega_d)$ we define an
$\sla(n+1,\bbR)$-valued 1-form
$$
{\mathcal A}=b^{-1}\bma\Gamma^a_{~b}-\frac{1}{n+1}\Gamma^c_{~c}\delta^a_{~b}&\theta^a\\\omega_b&-\frac{1}{n+1}\Gamma^c_{~c}\ema b+b^{-1}\der b.
$$  
This can be also written as 
$${\mathcal
  A}=\bma\hat{\Gamma}^a_{~b}-\frac{1}{n+1}\hat{\Gamma}^c_{~c}\delta^a_{~b}&\hat{\theta}^a\\\hat{\omega}_b&-\frac{1}{n+1}\hat{\Gamma}^c_{~c}\ema,$$
from which, knowing $b$, one can deduce the transformation rules
$$(\theta^a,\Gamma^b_{~c},\omega_d)\to
(\hat{\theta}^a,\hat{\Gamma}^b_{~c},\hat{\omega}_d),$$ 
see e.g. \cite{nn}. Note that when the coframe $\theta^a$ is fixed,
i.e. when $A^a_{~b}=\delta^a_{~b}$, these transformations coincide
with (\ref{pt}), (\ref{1tr}); the above setup extends these 
transformations to the situation when we allow the frame to change
under the action of the $\glg(n,\bbR)$ group.

The form $\mathcal A$ defines an $\sla(n+1,\bbR)$
Cartan connection on $H\to P\to M$. Its curvature
$${\mathcal R}=\der{\mathcal A}+{\mathcal A}\dz{\mathcal A},$$
satsifies 
$${\mathcal R}=b^{-1}\bma W^a_{~b}&0\\Y_b&0\ema b=\bma \hat{W}^a_{~b}&0\\\hat{Y}_b&0\ema,$$
and consists of the 2-forms $W^a_{~b}, Y_{b}$ as defined in
(\ref{cs}). In particular we have
$\hat{W}^a_{~b}=\tfrac12\hat{W}^a_{~bcd}\hat{\theta}^c\dz\hat{\theta}^d$,
and
$\hat{Y}_{a}=\tfrac12\hat{Y}_{abc}\hat{\theta}^b\dz\hat{\theta}^c$,
where $\hat{W}^a_{~bcd}$ and $\hat{Y}_{abc}$ are the transformed Weyl
and Cotton tensors.

Note that the $(n+n^2+n)$
1-forms $(\hat{\theta}^a,\hat{\Gamma}^b_{~c},\hat{\omega}_d)$
constitute a \emph{coframe} on the $(n^2+2n)$-dimensional bundle $H\to P\to M$; in particular these forms \emph{are linearly
independent} at each point of $P$. They satisfy the transformed Cartan
structure equations
\be\begin{aligned}
&\der\hat{\theta}^a+\hat{\Gamma}^a_{~b}\dz\hat{\theta}^b=0\\
&\der\hat{\Gamma}^a_{~b}+\hat{\Gamma}^a_{~c}\dz\hat{\Gamma}^c_{~b}=
\hat{W}^a_{~b}+\hat{\theta}^a\dz\hat{\omega}_b+\delta^a_{~b}\hat{\theta}^c\dz\hat{\omega}_c\\
&\der\hat{\omega}_a+\hat{\omega}_b\dz\hat{\Gamma}^b_{~a}=\hat{Y}_a.
\end{aligned}\label{cst}\ee 
\subsection{Fefferman metrics}
In Ref. \cite{ns}, with \emph{any} point equivalence class of second
order ODEs $y''=Q(x,y,y')$, we associated a certain 4-dimensional
manifold $P/\sim$ equipped with a \emph{conformal} class of
metrics of \emph{split} signature $[g_F]$, whose \emph{conformal
  invariants} encoded all the \emph{point invariants} of the ODEs from
the point equivalent class. By analogy with the theory of
3-dimensional CR structures we called the class $[g_F]$ the
\emph{Fefferman class}. The manifold $P$ from 
$P/\sim$ was a principal fiber bundle ${\mathcal
  H}\to P\to N$ over a \emph{three}-dimensional manifold $N$, which was
identified with the first jet space ${\mathcal J}^1$ of an ODE from
the equivalence class. The bundle $P$ was 8-dimensional, and
  ${\mathcal H}$ was a \emph{five}-dimensional parabolic subgroup of
$\slg(3,\bbR)$. For each point equivalnce class of ODEs
$y''=Q(x,y,y')$, the 
\emph{Cartan normal conformal connection} of the corresponding 
Fefferman metrics $[g_F]$, was reduced to a certain $\sla(3,\bbR)$
Cartan connection $\mathcal A$ on $P$. 
The \emph{two main components of the 
curvature} of this connection were the \emph{two classical basic point 
invariants} of the class $y''=Q(x,y,y')$, namely:
$$w_1=D^2Q_{y'y'} - 4DQ_{yy'} - DQ_{y'y'} Q_{y'} + 4Q_{y'}
  Q_{yy'} - 3Q_{y'y'} Q_y + 6Q_{yy},$$
and
$$ w_2=Q_{y'y'y'y'}.$$  
If \emph{both of these invariants were nonvanishing} the Cartan bundle that
encoded the structure of a point equivalence class of ODEs
$y''=Q(x,y,y')$ was just ${\mathcal H}\to{\mathcal P}\to N$ with the
Cartan connection $\mathcal A$. The nonvanishing of $w_1w_2$, was
reflected in the fact that the corresponding Fefferman metrics were
always of the Petrov type $N\times N'$, and never selfdual nor
antiselfdual. 

In case of $w_1w_2\equiv 0$, the
situation was more special \cite{nn}: the Cartan bundle ${\mathcal H}\to P\to N$
was also defining a Cartan bundle $H\to P\to M$, over a \emph{two}-dimensional
manifold $M$, with the \emph{six}-dimensional parabolic subgroup $H$
of $\slg(3,\bbR)$ as the structure group. The
manifold $M$ was identified with the solution space of an ODE
representing the point equivalent class. Furthermore the space $M$ was
naturally equipped with a \emph{projective structure} $[\nabla]$, whose
invariants were in one-to-one correspondence with the point invariants
of the ODE. This one-to-one correspondence was realized in terms of
the $\sla(3,\bbR)$ connection ${\mathcal A}$. This, although initially defined 
as a canonical $\sla(3,\bbR)$ connection on 
${\mathcal H}\to P\to N$, in the special case of $w_1w_2\equiv 0$ 
became the $\sla(3,\bbR)$-valued \emph{Cartan normal projective connection} of the structure 
$(M,[\nabla])$ on the Cartan bundle $H\to P\to M$. In such a case the
corresponding Fefferman class $[g_F]$ on $P/\sim$ became
\emph{selfdual} or \emph{antiselfdual} depending on which of the
invariants $w_1$ or $w_2$ vanished.

What we have overlooked in the discussions in \cite{nn,ns}, was that in the
case of $w_2\equiv 0$, $w_1\neq 0$ we could have defined \emph{two},
\emph{a priori different} Fefferman classes $[g_F]$ and $[g_F']$. As we see below 
the construction of these classes totally relies on the fact that 
we had a canonical projective structure $[\nabla]$ on $M$. Actually we
have the following theorem.

\begin{theorem}
Every $n$-dimensional manifold $M$ with a projective structure
$[\nabla]$ uniquely defines \emph{a number} $n$ of conformal metrics
$[g^a]$, each of \emph{split} signature $(n,n)$, and each 
defined on its own natural $2n$-dimensional subbundle $P_a=P/(\sim_a)$ of the 
Cartan projective bundle $H\to P\to M$.      
\end{theorem}
\begin{proof}
Given $(M,[\nabla])$ we will construct the pair $(P_a,[g^a])$ for each
$a=1,\dots,n$. We use the notation of Section \ref{ser}.

Let $(X_a,X^b_{~c},X^d)$ be a frame of vector fields on $P$ dual
to the coframe
$(\hat{\theta}^a,\hat{\Gamma}^b_{~c},\hat{\omega}_d)$. This means that
\be
X_a\hook \hat{\theta}^b=\delta^b_{~a},\quad\quad
X^a_{~b}\hook\hat{\Gamma}^c_{~d}=\delta^a_{~d}\delta^c_{~b},\quad\quad
X^a\hook\hat{\omega}_b=\delta^a_{~b},\label{hooks}\ee
at each point, with \emph{all} other contractions being zero. 

We now
define a number of $n$ bilinear forms $\hat{g}^a$ on $P$ defined by
$$\hat{g}^a=(\hat{\Gamma}^a_{~b}-\tfrac{2}{n+1}\hat{\Gamma}^c_{~c}\delta^a_{~b})\otimes
\hat{\theta}^b+\hat{\theta}^b\otimes(\hat{\Gamma}^a_{~b}-\tfrac{2}{n+1}\hat{\Gamma}^c_{~c}\delta^a_{~b}),$$
or
$$\hat{g}^a=2(\hat{\Gamma}^a_{~b}-\tfrac{2}{n+1}\hat{\Gamma}^c_{~c}\delta^a_{~b})\hat{\theta}^b,$$ 
for short. In this second formula we have used the classical notation, such
as for example in $g=g_{ab}\theta^a\theta^b$, which abreviates the symmetrized tensor product of two
1-forms $\lambda$ and $\mu$ on $P$ to $\lambda\otimes\mu+\mu\otimes\lambda=2\lambda\mu$. 

We note that the formula for $\hat{g}^a$, when written in terms of the
Cartan connection $\mathcal A$, reads\footnote{Compare with the defining formula
for $G$ in \cite{ns}}:
$$\hat{g}^a=2{\mathcal A}^a_{~\mu}{\mathcal A}^\mu_{~n+1},$$
where the index $\mu$ is summed over $\mu=1,\dots,n,n+1$. Indeed:
$$2{\mathcal A}^a_{~\mu}{\mathcal A}^\mu_{~n+1}=2(\hat{\Gamma}^a_{~b}-\tfrac{1}{n+1}\hat{\Gamma}^c_{~c}\delta^a_{~b})\hat{\theta}^b+2\hat{\theta}^a(-\tfrac{1}{n+1}\hat{\Gamma}^c_{~c})=2(\hat{\Gamma}^a_{~b}-\tfrac{2}{n+1}\hat{\Gamma}^c_{~c}\delta^a_{~b})\hat{\theta}^b=\hat{g}^a.$$
The bilinear forms $\hat{g}^a$ are \emph{degenerate} on $P$. For each
\emph{fixed} value
of the index $a$, $a=1,\dots,n$, they have $n^2$ degenerate directions
spanned by $(X^b,Z^c_{~D})$, where $b,c=1,\dots, n$ and $D=1,\dots, n$
\emph{without} $D=a$. The $n(n-1)$ vector fields $Z^c_{~D}$ are
defined to be
$$Z^c_{~D}=X^c_{~D}-\tfrac{2}{n-1}X^d_{~d}\delta^c_{~D}.$$  
Obviously $(X^b,Z^c_{~D})$ annihilate all $\theta^b$s. Also obviously
all $X^b$s annihilate all
$(\hat{\Gamma}^a_{~b}-\tfrac{1}{n+1}\hat{\Gamma}^c_{~c}\delta^a_{~b})$s. To
see that all $Z^c_{~D}$s annihilate all 
$(\hat{\Gamma}^a_{~b}-\tfrac{1}{n+1}\hat{\Gamma}^c_{~c}\delta^a_{~b})$s 
we extend the definition of $Z^c_{~D}$s to 
$$Z^c_{~f}=X^c_{~f}-\tfrac{2}{n-1}X^d_{~d}\delta^c_{~f},$$
where now $f=1,\dots, n$.
For these we get:
$$Z^c_{~d}\hook(\hat{\Gamma}^a_{~b}-\tfrac{2}{n+1}\hat{\Gamma}^h_{~h}\delta^a_{~b})=
\delta^c_{~b}\delta^a_{~d}.$$
Thus, if $d\neq a$ we see that each $Z^c_{~d}$ anihilates
$\hat{\Gamma}^a_{~b}-\tfrac{2}{n+1}\hat{\Gamma}^h_{~h}\delta^a_{~b}$. 
Hence $n(n-1)$ directions $Z^a_{~D}$ are degenarate directions for
$\hat{g}^a$.

Another observation is that the $n^2$ degenerate directions
$(X^b,Z^c_{~D})$ form an \emph{integrable} distribution. This is
simplest to see by considering their annihilator. At each point this
is spanned by the $2n$ one-forms
$(\hat{\theta}^b,\hat{\Gamma}^{(a)}_{~b}-\tfrac{2}{n+1}\hat{\Gamma}^h_{~h}\delta^{(a)}_{~b})$,
where the index $(a)$ in brackets says that it is a fixed $a$ which is
not present in the range of indices $D$. Now using (\ref{cst}) it is
straightforward to see that the forms
$(\hat{\theta}^b,\hat{\tau}^{(a)}_{~b})=(\hat{\theta}^b,\hat{\Gamma}^{(a)}_{~b}-\tfrac{2}{n+1}\hat{\Gamma}^h_{~h}\delta^{(a)}_{~b})$
satisfy the Frobenius condition
$$\der\hat{\theta}^a\dz\hat{\theta}^1\dz\dots\dz\hat{\theta}^n=0,$$
$$\der\hat{\tau}^{(a)}_{~b}\dz\hat{\tau}^{(a)}_{~1}\dz
\dots\dz\hat{\tau}^{(a)}_{~n}\dz\hat{\theta}^1\dz\dots\dz\hat{\theta}^n=0.$$
Thus the $n^2$-dimensional distribution spanned by $(X^b,Z^c_{~D})$ is
integrable.

Now, using (\ref{cst}) we calculate the Lie derivatives of $\hat{g}^a$
with respect to the directions $(X^b,Z^c_{~D})$. It is easy to see
that:
$${\mathcal L}_{X^b}\hat{g}^a=0$$
and 
$${\mathcal
  L}_{Z^c_{~d}}\hat{g}^a=-\delta^a_{~d}\hat{g}^c+\tfrac{2}{n-1}\delta^c_{~d}\hat{g}^a.$$
The last equation means also that
$${\mathcal
  L}_{Z^c_{~D}}\hat{g}^a=\tfrac{2}{n-1}\delta^c_{~D}\hat{g}^a.$$
Thus, the bilinear form $\hat{g}^a$ transforms \emph{conformally} when
Lie transported along the integrable distribution spanned by
$(X^b,Z^c_{~D})$. 

Now, for each fixed $a=1,\dots,n$, we introduce an equivalence
relation $\sim_a$ on $P$, which identifies points on
the same integral leaf of $\Span(X^b,Z^c_{~D})$. On the
$2n$-dimensional leaf space $P_a=P/{(\sim_a)}$ the $n^2$ degenerate
directions for $\hat{g}^a$ are squeezed to points. Since the remainder
of $\hat{g}^a$ is given up to a conformal rescalling on each leaf,
the bilinear form $\hat{g}^a$ \emph{descends} to a \emph{unique conformal
class} $[g^a]$ of metrics, which on $P_a$ have \emph{split signature}
$(n,n)$. Thus, for each $a=1,\dots,n$ we have constructed the
$2n$-dimensional split signature conformal structure
$(P_a,[g^a])$. It follows from the construction that $P_a$ may b
identified with \emph{any} $2n$-dimensional submanifold $\tilde{P}_a$
of $P$, which is \emph{transversal to the leaves of}
$\Span(X^b,Z^c_{~D})$. The conformal class $[g^a]$ is represented on
each $\tilde{P}^a$ by the \emph{restriction}
$g^a=\hat{g}^a_{~|\tilde{P}_a}$. This finishes the proof of the theorem.
\end{proof}
One can calculate the Cartan normal conformal connection for the
conformal structures  $(P_a,g^a)$. This is a lengthy, but
straightforward calculation. The result is given in the following
theorem.
\begin{theorem}
In the null frame $(\hat{\tau}^{(a)}_{~b},\hat{\theta}^c)$ the Cartan
normal conformal connection for the metric $\hat{g}^a$ is given by:
$$
G=\bma
-\tfrac{1}{n+1}\hat{\Gamma}^d_{~d}&0&-\hat{\omega}_c&0\\
&&&\\
\hat{\tau}^{(a)}_{~b}&-\hat{\Gamma}^e_{~b}+\tfrac{1}{n+1}\hat{\Gamma}^d_{~d}\delta^e_{~b}&\hat{\theta}^d\hat{R}^{(a)}_{~dcb}&-\hat{\omega}_b\\
&&&\\
\hat{\theta}^f&0&\hat{\Gamma}^f_{~c}-
\tfrac{1}{n+1}\hat{\Gamma}^d_{~d}\delta^f_{~c}&0\\
&&&\\
0&\hat{\theta}^e&\hat{\tau}^{(a)}_{~c}&\tfrac{1}{n+1}\hat{\Gamma}^d_{~d}
\ema.
$$
Its curvature $R=\der G+G\dz G$ is given by:
$$R=\bma
0&0&-\hat{Y}_c&0\\
&&&\\
0&-\hat{W}^e_{~b}&\hat{S}_{cb}&-\hat{Y}_b\\
&&&\\
0&0&\hat{W}^f_{~c}&0\\
&&&\\
0&0&0&0
\ema,$$
where 
$$
\begin{aligned}
&\hat{S}_{cb}=-\hat{\theta}^d(\hat{D}\hat{R}^{(a)}_{~dcb}-\hat{\tau}^{(a)}_{~s}\hat{W}^s_{~dcb})=\\
&-\hat{\theta}^d(\hat{D}\hat{W}^{(a)}_{~dcb}-\hat{\tau}^{(a)}_{~s}\hat{W}^s_{~dcb})+\delta^{(a)}_{~b}\hat{Y}_c-\delta^{(a)}_{~c}\hat{Y}_b.\end{aligned}$$
\end{theorem} 
\section{When a projective class includes a Levi-Civita connection?}\label{sec2}
\subsection{Projective structures of the Levi-Civita connection}\label{sec22}
Let us now assume that an $n$-dimensional manifold $M$ is equipped
with a (pseudo)\-Riemannian metric $\hat{g}$. We denote its 
\emph{Levi-Civita connection} by $\hat{\nabla}$. Levi-Civita
connection $\hat{\nabla}$ 
defines its \emph{projevtive class} $[\nabla]$ 
with connections $\nabla$ such that (\ref{prst}) holds. Now, with the
Levi-Civita representative $\hat{\nabla}$ of $[\nabla]$ we can define
its curvature $\hat{\Omega}^a_{~b}$, as in (\ref{c12}), and decompose
it onto the \emph{projective} Weyl and Schouten tensors
$\hat{W}^a_{~bcd},\hat{\Rho}_{ab}$, as in
(\ref{c13}):
\be
\hat{\Omega}^a_{~b}=\hat{W}^a_{~b}+\theta^a\dz\hat{\omega}_b+\delta^a_{~b}\theta^c\dz\hat{\omega}_c.\label{cf1}\ee
However, since now $M$ has an additional metric structure
$\hat{g}=\hat{g}_{ab}\theta^a\theta^b$, with the inverse
$\hat{g}^{ab}$ such that $\hat{g}_{ab}\hat{g}^{bc}=\delta^c_{a}$,
another decomposition of the curvature is possible. This is the 
decomposition onto the \emph{metric} Weyl and Schouten tensors
$\weyl^a_{~bcd}, \schou_{ab}$, given by:
\be
\hat{\Omega}^a_{~b}=\weyl^a_{~b}+\hat{g}^{ac}\hat{g}_{bd}\oms_c\dz\theta^d+\theta^a\dz\oms_b.\label{cf2}\ee
The tensor counterparts of the formulae (\ref{cf1})-(\ref{cf2}) are
respectively:
\be
\begin{aligned}
&\hat{R}^a_{~bcd}=\hat{W}^a_{~bcd}+\delta^a_{~c}\hat{\Rho}_{db}-\delta^a_{~d}\hat{\Rho}_{cb}-
2\delta^a_{~b}\hat{\Rho}_{[cd]}\\
&\hat{R}^a_{~bcd}=\weyl^a_{~bcd}+\delta^a_{~c}\schou_{db}-\delta^a_{~d}\schou_{cb}
+\hat{g}_{bd}\hat{g}^{ae}\schou_{ec}-\hat{g}_{bc}\hat{g}^{ae}\schou_{ed}.
\end{aligned}
\label{incop}\ee
To find relations between the projective and the metric Weyl and
Schouten tensors one compares the r.h. sides of
(\ref{incop}). For example, because of the equality on the left hand
sides of (\ref{incop}), the projective and the Levi-Civita Ricci
tensors are equal:
$$\hat{R}_{bd}=\hat{R}^a_{~bad}=\rlc_{bd}.$$
Thus, via (\ref{c41}), we get
\be
\hat{\Rho}_{ab}=\tfrac{1}{n-1}\rlc_{ab}.\label{rholc}\ee
Further relations between the projective and Levi-Civita objects can
be obtained by recalling that:  
$$\rlc_{ab}=(n-2)\schou_{ab}+\hat{g}_{ab}\schou,$$
where
$$\schou=\hat{g}^{ab}\schou_{ab},$$
and that the Levi-Civita Ricci scalar is given by:
$$\rlc=\hat{g}^{ab}\rlc_{ab}.$$
After some algebra we get the following proposition.
\begin{proposition}\label{pr211} 
The projective Schouten tensor $\hat{\Rho}_{ab}$ for the
Levi-Civita connection $\hat{\nabla}$ \emph{is related} to the metric
Schouten tensor $\schou_{ab}$ via:
$$\hat{\Rho}_{ab}=\schou_{ab}-\frac{1}{(n-1)(n-2)}G_{ab},$$
where $G_{ab}$ is the \emph{Einstein tensor} for the Levi-Civita 
connection:
$$G_{ab}=\rlc_{ab}-\tfrac12 \hat{g}_{ab}\rlc.$$
The projective Weyl tensor $\hat{W}^a_{~bcd}$ for the
Levi-Civita connection $\hat{\nabla}$ is related to 
the metric Weyl tensor $\weyl^a_{~bcd}$ via:
\be
\begin{aligned}
&\hat{W}^a_{~bcd}=\weyl^a_{~bcd}~+~\frac{1}{n-2}(\hat{g}_{bd}\hat{g}^{ae}\rlc_{ec}-\hat{g}_{bc}\hat{g}^{ae}\rlc_{ed})~+\\
&\frac{1}{(n-1)(n-2)}(\delta^a_{~c}\rlc_{db}-\delta^a_{~d}\rlc_{cb})~+~
\frac{\rlc}{(n-1)(n-2)}(\delta^a_{~d}\hat{g}_{bc}-\delta^a_{~c}\hat{g}_{bd}).\end{aligned}\label{ww}\ee
\end{proposition}
In particular we have the following corollary:
\begin{corollary}\label{co222}
The projective Schouten tensor $\hat{\Rho}_{ab}$ of the Levi-Civita
connection $\hat{\nabla}$ is symmetric
$$\hat{\Rho}_{ab}=\hat{\Rho}_{ba}.$$
Moreover, the projective Weyl tensor $W^a_{~bcd}$ of \emph{any}
connection $\nabla$ from the projective class $[\nabla]$ of a 
Levi-Civita connection
satisfies
\be
\hat{g}_{ae}\hat{g}^{bc}W^e_{~bcd}=
\hat{g}_{de}\hat{g}^{bc}W^e_{~bca}.\label{icc}\ee
\end{corollary}
\begin{proof}
The first part of the Corollary is an immediate consequence of 
the fact that the metric Schouten tensor
of the Levi-Civita connection as well as the Einstein tensor are symmetric. The second part follows
from the relation (\ref{ww}), which yields:
$$(n-1)\hat{g}_{ae}\hat{g}^{bc}\hat{W}^e_{~bcd}=-n\rlc_{ad}+\rlc \hat{g}_{ad}.$$
Since $\rlc_{ab}$ is symmetric we get $\hat{g}_{ae}\hat{g}^{bc}\hat{W}^e_{~bcd}=
\hat{g}_{de}\hat{g}^{bc}\hat{W}^e_{~bca}$. But according to the
fourth transformation law in (\ref{2tr}) the Weyl tensor is invariant
under the projective transformations,
$\hat{W}^a_{~bcd}=W^a_{~bcd}$. Thus (\ref{icc}) holds, for all
connections $\nabla$ from the projective class of $\hat{\nabla}$. 
This ends the proof.
\end{proof}
The above corollary is obviously related to the question in the title
of this Section. It gives the first, very simple, obstruction for a
projective structure $[\nabla]$ to include a Levi-Civita connection of
some metric. We reformulate it to the following theorem.
\begin{theorem}\label{so1}
A necessary condition for a projective structure $(M,[\nabla])$ to
include a connection $\hat{\nabla}$, which is the Levi-Civita
connection of some metric $\hat{g}_{ab}$, is an existence of a symmetric
nondegenerate bilinear form $g^{ab}$ on $M$, such that the Weyl tensor
$W^a_{~bcd}$ of the projective structure satisfies
\be
g_{ae}g^{bc}W^e_{~bcd}=
g_{de}g^{bc}W^e_{~bca},\label{cu}\ee
with $g_{ab}$ being the inverse of $g^{ab}$,
$g_{ac}g^{cb}=\delta^b_{~a}$. If the 
Levi-Civita connection  $\hat{\nabla}$  from the projective class
$[\nabla]$ exists, then its corresponding metric $\hat{g}_{ab}$ must be 
\emph{conformal} to the inverse $g_{ab}$ of some solution $g^{ab}$ of equation (\ref{cu}), i.e. 
$\hat{g}_{ab}={\rm e}^{2\phi}g_{ab}$, for a solution $g^{ab}$ of (\ref{cu})
and some function $\phi$ on $M$.      
\end{theorem} 
As an example we consider a projective structure $[\nabla]$ on a 3-dimensional
manifold $M$ parametrized by three real coordinates $(x,y,z)$. We choose a
holonomic coframe $(\theta^1,\theta^2,\theta^3)=(\der x,\der y,\der
z)$, and generete a projective structure from the
connection 1-forms 
\be
\Gamma^a_{~b}=\bma 0&a\der z&a\der y\\
b\der z&0 &b\der x\\
c\der y&c\der x&0\ema ,\quad\quad{\rm with}\quad\quad a=a(z),\quad
b=b(z),\quad c=c(z),\label{exc1}
\ee
via (\ref{pt}).

It is easy to calculate the projective Weyl forms
$W^a_{~b}$, and the projective Schouten forms $\om_b$, for this
connection. They read:
$$W^a_{~b} =\bma -\tfrac12 c'\der x\dz\der y&0&-a'\der y\dz\der z\\
0&\tfrac12 c'\der x\dz\der y&-b'\der x\dz\der z\\
-\tfrac12 c'\der y\dz\der z&-\tfrac12 c'\der x\dz\der z&0
\ema,$$
and
$$\omega_a=\bma -bc\der x+\tfrac12 c'\der y,&-ac \der y+\tfrac12 c'\der
x,&-ab\der z\ema.$$
With this information in mind it is easy to check that 
\be
g^{ab}=\bma-f a'&g^{12}&0\\g^{12}&-f b'&0\\0&0&g^{33}
\ema,\label{exce}\ee
with some undetermined functions $f=f(x,y,z)$,
$g^{12}=g^{12}(x,y,z)$, $g^{33}=g^{33}(x,y,z)$, satisfies (\ref{cu}). Thus the connection $\Gamma^a_{~b}$ may, in
principle, be the Levi-Civita connection of some metric $\hat{g}_{ab}$. Accordig to
Theorem \ref{so1} we may expect that the inverse of this $g^{ab}$ is proportional to $\hat{g}_{ab}$. 
\subsection{Comparing natural projective and (pseudo)Riemannian tensors}
Proposition \ref{pr211} in an obvious way implies the following
corollary:
\begin{corollary}
The Levi-Civita connection $\hat{\nabla}$ of a metric $\hat{g}_{ab}$
has its projective Schouten tensor equal to the Levi-Civita one,
$\hat{\Rho}_{ab}=\schou_{ab}$, if and only if its Einstein (hence the
Ricci) tensor vanishes. If this happens $\hat{\Rho}_{ab}\equiv 0$, and
both the projective and the Levi-Civita Weyl tensors are equal,
$\hat{W}^a_{~bcd}=\weyl^a_{~bcd}$. 
\end{corollary} 
Now we answer the question if there are Ricci non-flat metrics having
equal projective an Levi-Civita Weyl tensors. We use (\ref{ww}). The
requirement that $\hat{W}^a_{~bcd}=\weyl^a_{~bcd}$ yields the
following proposition.
\begin{proposition}
The Levi-Civita connection $\hat{\nabla}$ of a metric $\hat{g}_{ab}$
has its projective Weyl tensor equal to the Levi-Civita one,
$\hat{W}^a_{~bcd}=\weyl^a_{~bcd}$, if and only if its Levi-Civita Ricci
tensor satisfies
\be M_{abcd}^{\quad\,\,\,ef}~\rlc_{ef}=0,\label{cwe}\ee
where
$$M_{abcd}^{\quad\,\,\,ef}=\hat{g}_{ac}\delta^e_{~d}\delta^f_{~b}-\hat{g}_{ad}\delta^e_{~c}\delta^f_{~b}+\hat{g}_{ad}\hat{g}_{cb}\hat{g}^{ef}-\hat{g}_{ac}\hat{g}_{db}\hat{g}^{ef}+(n-1)(\hat{g}_{bd}\delta^e_{~a}\delta^f_{~c}-\hat{g}_{bc}\delta^e_{~a}\delta^f_{~d}).$$
\end{proposition} 
One easilly checks that the \emph{Einstein metrics}, i.e. the metrics for which 
$$\rlc_{ab}=\Lambda \hat{g}_{ab},$$ satisfy
(\ref{cwe}). Therefore we have the following corollary:
\begin{corollary}\label{wni}
The projective and the Levi-Civita Weyl tensors of \emph{Einstein}
metrics are equal. In particular, all conformally flat Einstein
metrics (metrics of constant curvature) are projectively equivalent.  
\end{corollary}
It is interesting to know if there are non-Einstein metrics satisfying
condition (\ref{cwe}). 
\subsection{Formulation a'la Roger Liouville}
If $\nabla$ is in the projective class of the Levi-Civita connection
$\hat{\nabla}$ of a metric $\hat{g}$ we have: 
$$0=\hat{\D}\hat{g}_{ab}=\D\hat{g}_{ab}-2A\hat{g}_{ab}-A_a\theta^c\hat{g}_{cb}-A_b\theta^c\hat{g}_{ac},$$
for some 1-form $A=A_a\theta^a$. Thus the condition that a 
torsionless connection $\nabla$ is projectively equivalent to the 
Levi-Civita connection of some metric, is equivalent to the existence
of a pair $(\hat{g}_{ab},A_a)$ such that 
$$\D\hat{g}_{ab}=2A\hat{g}_{ab}+\theta^c(A_a\hat{g}_{cb}+A_b\hat{g}_{ac}),$$
with an \emph{invertible} symmetric tensor $\hat{g}_{ab}$. Dually this
last means that  
a 
torsionless connection $\nabla$ is projectively equivalent to a
Levi-Civita connection of some metric, iff there exists 
a pair $(\hat{g}^{ab},A_a)$ such that
\be 
\D\hat{g}^{ab}=-2A\hat{g}^{ab}-A_c(\theta^b\hat{g}^{ca}+\theta^a\hat{g}^{cb}),\label{obc}
\ee
with an invertible $\hat{g}^{ab}$. 

The unknown $A$ can be easilly eliminated from these equations by
contracting with the inverse $\hat{g}_{ab}$:
$$A=-\frac{\hat{g}_{ab}\D\hat{g}^{ab}}{2(n+1)},$$
so that the `if an only if' condition for $\nabla$ to be in a
projective class of a Levi-Civita connection $\hat{\nabla}$ is the 
existence of $\hat{g}^{ab}$ such that 
$$2(n+1)\D\hat{g}^{ab}=2(\hat{g}_{cd}\D\hat{g}^{cd})\hat{g}^{ab}+(\hat{g}_{ef}\nabla_c\hat{g}^{ef})(\theta^b\hat{g}^{ca}+\theta^a\hat{g}^{cb}),\quad\quad
\hat{g}_{ac}\hat{g}^{cb}=\delta^b_{~c}.$$
This is an unpleasent to analyse, \emph{nonlinear} system of PDEs,  
for the unknown $\hat{g}^{ab}$. It follows that it is more convenient
to discuss the equivalent system (\ref{obc}) for the unknowns $(\hat{g}^{ab},A_a)$, which we will do in the following. 

The aim of this subsection is to prove the following theorem:
\begin{theorem}\label{tth}
A torsionless connection $\tilde{\nabla}$ on an $n$-dimensional
manifold $M$ is projectively equivalent to a
Levi-Civita connection $\hat{\nabla}$ of a metric $\hat{g}_{ab}$ if
and only if its projective class $[\tilde{\nabla}]$ contains a
\emph{special} projective subclass $[\nabla]$ whose connections $\nabla$ satisfy the
following:
for every $\nabla\in[\nabla]$ there exists a nondegenerate symmetric  tensor $g^{ab}$ and a vector field
$\mu^a$ on $M$ such that  
$$\nabla_c g^{ab}=\mu^a\delta^b_{~c}+\mu^b\delta^a_{~c},$$
or what is the same:
\be 
\D g^{ab}=\mu^a\theta^b+\mu^b\theta^a.\label{obc1}\ee   
\end{theorem}

\begin{proof}
If $\hat{\nabla}$ is the Levi-Civita connection of a metric
$\hat{g}=\hat{g}_{ab}\theta^a\theta^b$, we consider 
connections $\nabla$ associated with $\hat{\nabla}$ via (\ref{prst}),
in which $A=\der\phi$, with arbitrary functions (potentials) on
$M$. This is a \emph{special} class of connections, since the
projective Schouten tensor $\hat{\Rho}_{ab}$ for $\hat{\nabla}$ is 
\emph{symmetric} (see Corollary \ref{co222}), and the transformation
(\ref{2tr}) with gradient $A$s, preserves the symmetry of the
projective Schouten tensor (see Remark \ref{ko}). 

Any connection $\nabla$ from this special class satisfies (\ref{obc})
with $A=\der \phi$, and therefore is characterized by the potential
$\phi$, $\nabla=\nabla(\phi)$.
 
We now take the inverse $\hat{g}^{ab}$ of the metric $\hat{g}^{ab}$,
$\hat{g}_{ac}\hat{g}^{cb}=\delta^b_{~c}$, and \emph{rescale} it to 
$$g^{ab}={\rm e}^{2f}\hat{g}^{ab},$$
where $f$ is a function on $M$. Using (\ref{obc}) with $A=\der\phi$, after a short
algebra, we get: 
$$\D g^{ab}=-2(\der\phi-\der f)g^{ab}-(\nabla_c\phi)(\theta^bg^{ca}+\theta^ag^{cb}).$$ 
Thus taking $$f=\phi+ {\rm const},$$ for each $\nabla=\nabla(\phi)$
from the special class $[\nabla]$, we associate $g^{ab}={\rm
  e}^{2f}\hat{g}^{ab}$ satisfying 
$$\D g^{ab}=-(\nabla_c\phi)(\theta^bg^{ca}+\theta^ag^{cb}).$$ 
Defining $\mu^a=-A_c g^{ca}=-{\rm e}^{2f}(\nabla_c\phi)\hat{g}^{ca}$ we get
(\ref{obc1}). Obviously $g^{ab}$ is symmetric and nondegenerate since 
$\hat{g}^{ab}$ was. 

The proof in the opposite direction is as follows:

We start with $(\nabla,g^{ab},\mu^a)$ satisfying (\ref{obc1}). In
particular, connection $\nabla$ is \emph{special}, i.e. it has symmetric
projective Schouten tensor and, by Corollary (\ref{kor}), its curvature satisfies
$$\Omega^a_{~a}=0.$$ 

Since $g^{ab}$
is invertible, we have a symmetric $g_{ab}$ such that
$g_{ac}g^{cb}=\delta^b_{~a}$. We define 
\be
A=-g_{ab}\mu^b\theta^a.\label{aa}\ee
Contracting with (\ref{obc1}) we get:
$$g_{ab}\D g^{ab}=-2A,\quad\quad {\rm or}\quad\quad A=-\tfrac12
g_{ab}\D g^{ab}.$$
Now this last equation implies that:
$$\der A=-\tfrac12\D g_{ab}\dz\D g^{ab}-\tfrac12 g_{ab}\D^2 g^{ab}.$$
This compared with the Ricci identity $\D^2g^{ab}=\Omega^a_{~c}g^{cb}+\Omega^b_{~c}g^{ac}$, the defining equation 
(\ref{obc}), and its dual 
$$\D g_{ab}=-g_{ac}g_{bd}(\mu^c\theta^d+\mu^d\theta^c),$$
yields $$\der A=-\Omega^a_{~a}=0.$$
Thus the 1-form $A$ defined by (\ref{aa}) is \emph{locally a gradient}
of a function $\phi_0$ on $M$, $A=\der\phi_0$.  
The potential $\phi_0$ is defined by $(\nabla,g^{ab},\mu^a)$ up to
$\phi_0\to\phi=\phi_0+{\rm const}$, $$A=\der\phi.$$ 
We use it to rescale the inverse 
$g_{ab}$ of $g^{ab}$. We define 
$$\hat{g}_{ab}={\rm e}^{2\phi}g_{ab}.$$
This is a nondegenerate symmetric tensor on $M$. 

Using our definitions we finally get
$$\begin{aligned}
&\D \hat{g}_{ab}=\\
&2\der\phi\hat{g}_{ab}-{\rm
    e}^{2\phi}g_{ac}g_{bd}(\mu^c\theta^d+\mu^d\theta^c)=\\
&2A\hat{g}_{ab}+A_a\hat{g}_{bc}\theta^c+A_b\hat{g}_{ac}\theta^c.
\end{aligned}
$$
This means that the new torsionless connection $\hat{\nabla}$ defined
by (\ref{prst}), with $A$ as above, satisfies  
$$\hat{\D}\hat{g}_{ab}=\D
\hat{g}_{ab}-2A\hat{g}_{ab}-A_a\hat{g}_{bc}\theta^c-A_b\hat{g}_{ac}\theta^c=0,$$ 
and thus is the Levi-Civita connection
for a metric $\hat{g}=\hat{g}_{ab}\theta^a\theta^b$. Since
$A=\der\phi$ this shows that in the special projective class defined
by $\nabla$ there is a Levi-Civita connection $\hat{\nabla}$. This
finishes the proof.  
\end{proof} 
 We also have the following corollary, which can be traced back to 
Roger Liouville \cite{liu}, (see also \cite{bde,ema,mik,sin}):  
\begin{corollary}
A projective structure $[\hat{\nabla}]$ on $n$-dimensional manifold $M$
contains a Levi-Civita connection of some 
metric if and only if at least one special connection $\nabla$ in
$[\hat{\nabla}]$  
admits a solution to the equation 
\be
\nabla_cg^{ab}-\frac{1}{n+1}\delta^a_{~c}\nabla_dg^{bd}-\frac{1}{n+1}\delta^b_{~c}\nabla_dg^{ad}=0.\label{mai}
\ee
with a symmetric and nondegenerate tensor $g^{ab}$.
\end{corollary} 
\begin{proof}
We use Theorem \ref{tth}. 

If $(\nabla,g^{ab},\mu^a)$ satisfies
(\ref{obc1}) it is a simple calculation to show that (\ref{mai})
holds. 

The other way around: if (\ref{mai}) holds for a special connection
$\nabla$ and an invertible $g^{ab}$, then defining $\mu^a$ by 
$\mu^a=\frac{1}{n+1}\nabla_dg^{ad}$ we get $\nabla_c
g^{ab}=\mu^a\delta^b_{~c}+\mu^b\delta^a_{~c}$, i.e. the equation
(\ref{obc1}), 
after contracting with $\theta^c$. Now, if we take any other special
connection $\hat{\nabla}$, then it is related to 
$\nabla$ via
$\hat{\nabla}_X(Y)=\nabla_X(Y)+X(\phi)Y+Y(\phi)X$. Rescalling the
$g^{ab}$ to 
$\hat{g}^{ab}={\rm e}^{-2\phi}g^{ab}$ one checks that 
$\hat{\nabla}_c\hat{g}^{ab}-\frac{1}{n+1}\delta^a_{~c}\hat{\nabla}_d\hat{g}^{bd}-\frac{1}{n+1}\delta^b_{~c}\hat{\nabla}_dg^{ad}=0$. Thus
in any special connection $\hat{\nabla}$ we find an invertible
$\hat{g}^{ab}={\rm e}^{-2\phi}g^{ab}$ with
$\hat{\mu}^a=\tfrac{1}{n+1}\hat{\nabla}_d\hat{g}^{ad}$ satisfying
$\hat{\nabla}_c\hat{g}^{ab}=\hat{\mu}^a\delta^b_{~c}+\hat{\mu}^b\delta^a_{~c}$. 
\end{proof}
\begin{remark}
It is worthwhile to note that $\mu^a$ and $\mu^b$ as in the above
proof are realte by
$$\hat{\mu}^a={\rm
  e}^{-2\phi}(\mu^a+ g^{da}\nabla_d\phi).$$
\end{remark}

\subsection{Prolongation and obstructions}
In this section, given a projective structure $[\nabla]$, we restrict
it to a corresponding \emph{special} projective subclass. All the
calculations below, are performed assuming that $\nabla_a$ is in this
special projective subclass.

We will find consequences of the neccessary and sufficient conditions
(\ref{obc1}) for this special class to include a Levi-Civita
connection.

Applying $\D$ on both sides  of (\ref{obc1}), and using the Ricci
identity (\ref{ri}) we get as a consequence:
\be
\Omega^b_{~a}g^{ac}+\Omega^c_{~a}g^{ba}=\D\mu^c\dz\theta^b+\D\mu^b\dz\theta^c.\label{pa}\ee
This expands to the following tensorial equation:
\be
\delta^b_{~d}\nabla_a\mu^c-\delta^b_{~a}\nabla_d\mu^c+\delta^c_{~d}\nabla_a\mu^b-\delta^c_{~a}\nabla_d\mu^b=R^b_{~ead}g^{ec}+R^c_{~ead}g^{be}.\label{paw}\ee
Now contracting this equation in $\{ac\}$ we get:
\be
\nabla_a\mu^b=\delta^b_{~a}\rho-\Rho_{ac}g^{bc}-\frac{1}{n}W^b_{~cda}g^{cd}
\ee
with some function $\rho$ on $M$. This is the prolonged equation
(\ref{obc1}). It can be also written as:
\be
\D\mu^b=\rho\theta^b-\omega_cg^{bc}-\frac{1}{n}W^b_{~cda}g^{cd}\theta^a.\label{wai}
\ee
Applying $\D$ on both sides of this equation, after some
manipulations, one gets the equation for the function $\rho$:
\be
\nabla_a\rho=-2\Rho_{ab}\mu^b+\frac{2}{n}Y_{abc}g^{bc}.\ee 
This is the last prolonged equation implied by (\ref{obc1}). It can be
also written as:
\be
\D\rho=-2\omega_b\mu^b+\frac{2}{n}Y_{abc}g^{bc}\theta^a.\label{vai}\ee

Thus we have the following thoerem \cite{ema}:
\begin{theorem}\label{sto}
The equation (\ref{mai}) admits a solution for $g^{ab}$ if and
only if the following system
\be\begin{aligned}
&\D g^{bc}=\mu^c\theta^b+\mu^b\theta^c\\
&\D\mu^b=\rho\theta^b-\omega_cg^{bc}-\frac{1}{n}W^b_{~cda}g^{cd}\theta^a\\
&\D\rho=-2\omega_b\mu^b+\frac{2}{n}Y_{abc}g^{bc}\theta^a,
\end{aligned}\label{sy}
\ee
has a solution for $(g^{ab},\mu^c,\rho)$. 
\end{theorem}
Simple obstructions for having solutions to (\ref{sy}) are obtained by
inserting $\D\mu^b$ from (\ref{wai}) into the integrability conditions
(\ref{pa}), or what is the same, into (\ref{paw}). This insertion, after some algebra,
yields the following proposition.
\begin{proposition}\label{sto1}
Equation (\ref{wai}) is compatible with the integrability conditions
(\ref{pa})-(\ref{paw}) only if $g^{ab}$ satisfies the following
\emph{algebraic} equation:
\be
T_{[ed]}^{\quad cb}\phantom{}_{af}g^{af}=0,\label{equi}\ee
where 
\be\begin{aligned}
T_{[ed]}^{\quad cb}\phantom{}_{af}=\tfrac12 \delta^c_{~(a}W^b_{~f)ed}+\tfrac12 \delta^b_{~(a}W^c_{~f)ed}+\tfrac{1}{n}W^c_{~(af)[e}\delta^b_{~d]}+\tfrac{1}{n}W^b_{~(af)[e}\delta^c_{~d]}.
\end{aligned}\label{equiv}\ee  
\end{proposition}
\begin{remark}\label{lpo}
Note that although the integrability condition (\ref{equi}) was
derived in the special gauge when the connection $\nabla$ was special,
it is gauge inedependent. This is because the condition involves the
\emph{projectively invariant} Weyl tensor, and because it is homogeneous in
$g^{ab}$.   
\end{remark}
For each pair of distinct indices $[ed]$ the tensor $T_{[ed]}^{\quad cb}\phantom{}_{af}$ provides a map
\be
S^2M\ni \kappa^{ab}\stackrel{{\mathcal T}_{[ed]}}{\longrightarrow} \kappa'^{ab}=T_{[ed]}^{\quad
  ab}\phantom{}_{cd}\kappa^{cd}\in S^2M,\label{mct}\ee
which is an endomorphism ${\mathcal T}_{[ed]}$ of the space $S^2M$ of 
symmetric 2-tensors on $M$. It is therefore clear that equation (\ref{equi})
has a \emph{nonzero} solution for $g^{ab}$ only if \emph{each} of these
endomorphisms \emph{is singular}. Therefore we have the following
theorem (see also the last Section in \cite{bde}):
\begin{theorem}\label{sto2}
A neccessary condition for a projective structure $[\nabla]$ to
include a Levi-Civita connection of some metric $g$ is that all the
endomorphisms ${\mathcal T}_{[ed]}:S^2M\to S^2M$, built from its Weyl
tensor, as in
(\ref{equiv}), have nonvanishing
determinants. In dimension $n\geq 3$ this gives in general $\frac{n(n-1)}{2}$
obstructions to metrisability.  
\end{theorem}
\begin{remark}
{\bf Puzzle}: Note that here we have $I=\frac{n(n-1)}{2}$ obstructions, wheras
the naive count, as adapted from \cite{bde}, yields 
$I'=\tfrac14(n^4-7n^2-6n+4)$. For $n=3$ we see that we constructed
$I=3$ invariants, wheres $I'$ says that there is only one. Why?
\end{remark}
\begin{remark}
Note that the Remark \ref{lpo} enabled us to use \emph{any} connection
from the projective class, not only the special ones, in this theorem.  
\end{remark} 
Further integrability conditions for (\ref{obc1}) may be obtained by applying $\D$ on both sides
of (\ref{wai}) and (\ref{vai}). Applying it on (\ref{wai}), using again the Ricci identity
(\ref{ri}), after some algebra, we get the following proposition.
\begin{proposition}\label{sto3}
The integrability condition $\D^2\mu^b=\Omega^b_{~a}\mu^a$, for
$(g^{ab},\mu^c,\rho)$ satsifying (\ref{sy}), is equivalent to:
\be S_{[ae]}^{\quad \,\, b}\phantom{}_{cd}g^{cd}=\big(~
\frac{n+4}{2} W^b_{~cae}+W^b_{~[ae]c}~\big)\mu^c,\label{ic1}\ee
where the tensor $S_{[ae]}^{\quad \,\, b}\phantom{}_{cd}$ is given by:
$$S_{[ae]}^{\quad \,\, b}\phantom{}_{cd}=\frac{n-2}{2}Y_{ea(c}\delta^b_{~d)}+\nabla_{(c}W^b_{~d)ea}+W^b_{~(cd)[e;a]}.$$
Here, in the last term, for simplicity of the notation, we have used
the semicolon to denote the covariant derivative, $\nabla_ef=f_{;e}$. 
\end{proposition}  
\begin{remark}
Note that in dimension $n=2$, where $W^a_{~bcd}\equiv 0$, the 
inetrgrability conditions (\ref{equi}) and (\ref{ic1}) are automatically satisfied.
\end{remark}
The last integrability condition $\D^2\rho=0$ yields:
\begin{proposition}\label{sto4}
The integrability condition $\D^2\rho=0$, for
$(g^{ab},\mu^c,\rho)$, satsifying (\ref{sy}) is equivalent to:
\be
U_{[ab]cd}g^{cd}=-\frac{n+3}{2}Y_{bac}\mu^c,\label{i2}
\ee
where the tensor $U_{[ab](cd)}$ reads:
$$U_{[ab]cd}=\nabla_{[a}Y_{b](cd)}+W^e_{~(cd)[a}\Rho_{b]e}.$$
\end{proposition} 
\begin{remark}
For the sufficciency of conditions (\ref{equi}), (\ref{ic1}) and
(\ref{i2}) see Remark \ref{bull}.
\end{remark}

\section{Metrisability of a projective structure check list}\label{bu}
Here, based on Theorems \ref{so1}, \ref{tth}, \ref{sto}, \ref{sto2}
and Propositions \ref{sto1}, \ref{sto3} and \ref{sto4}, we outline a
procedure how to check if a given projective structure contains a
Levi-Civita connection of some metric. The procedure is valid for the
dimension $n\geq 3$.

Given a projective structure $(M,[\nabla])$ on an $n$-dimensional
manifold $M$:
\begin{enumerate}
\item calculate its Weyl tensor $W^a_{~bcd}$ and the corresponding
  operators ${\mathcal T}_{[ed]}$ as in (\ref{mct}). If at least one
  of the determinants $\tau_{ed}=\det({\mathcal T}_{[ed]})$,
  $e<d=1,2,\dots,n$, is not zero the projective structure
  $(M,[\nabla])$ does \emph{not} include any Levi-Civita connection.
\item If all the determinats  $\tau_{ed}$ vanish, find a special
  connection $\nabla^0$ in $[\nabla]$, and restrict to a special
  projective subclass $[\nabla^0]\subset[\nabla]$.
\item Now taking any connection $\nabla$ from $[\nabla^0]$ calculate
  the Weyl, (symmetric) Schouten, and Cotton tensors, and the tensors
  $T_{[ed]}^{\quad cb}\phantom{}_{af}$, $S_{[ae]}^{\quad \,\, b}\phantom{}_{cd}$, $U_{[ab]cd}$ of
  Propositions \ref{sto1}, \ref{sto3} and \ref{sto4}. 
\item\label{bui} Solve the \emph{linear algebraic} equations (\ref{equi}),
  (\ref{ic1}) and (\ref{i2}) for the unknown symmetric tensor $g^{ab}$
  and vector field $\mu^a$. 
\item If these equations \emph{have no solutions}, or the $n\times n$
  symmetric matrix
  $g^{ab}$ \emph{has vanishing determinant}, then $(M,[\nabla])$ does \emph{not} include any Levi-Civita connection.
\item If equations (\ref{equi}), (\ref{ic1}) and (\ref{i2}) admit
  solutions with nondegenerate $g^{ab}$, find the inverse $g_{ab}$ of
  the general solution for $g^{ab}$, and check if equation (\ref{icc})
  is satisfied. If this equation can not be satisfied by restricting
  the free functions in the general solution $g^{ab}$ of equations (\ref{equi}),
  (\ref{ic1}) and (\ref{i2}), then $(M,[\nabla])$ does \emph{not} include any Levi-Civita connection.   
\item \label{buj}In the opposite case restrict the general solution $g^{ab}$ of (\ref{equi}),
  (\ref{ic1}) and (\ref{i2}) to $g^{ab}$s satsifying (\ref{icc}), and
  insert $(g^{ab}, \mu^a)$, with such $g^{ab}$ and the most general $\mu^a$
  solving (\ref{equi}),
  (\ref{ic1}) and (\ref{i2}), in the equations
  (\ref{sy}). 
\item \label{buk}Find the general solution to the equations (\ref{sy}) for
  $(g^{ab},\mu^a,\rho)$, with $(g^{ab},\mu^a)$ from the ansatz
  described in point (\ref{buj}).
\item If the solution for such $(g^{ab},\mu^a,\rho)$ does not
  exist, or the symmetric tensor $g^{ab}$ is degenerate, then
  $(M,[\nabla])$ does \emph{not} include any Levi-Civita connection.
\item \label{bum}Otherwise find the inverse $g_{ab}$ of $g^{ab}$ from the
  solution $(g^{ab},\mu^a,\rho)$, and solve for a function $\phi$ on
  $M$ such that
  $\der\phi=-g_{ab}\mu^a\theta^b$. 
\item The metric $\hat{g}={\rm e}^{2\phi}g_{ab}\theta^a\theta^b$ has
  the Levi-Civita connection $\hat{\nabla}$ which is in the special
  projective class $[\nabla^0]\subset[\nabla]$. 
\end{enumerate}

\section{Three dimensional examples}
{\bf Example 1.}
Here, as the first example, we consider a 3-dimensional
projective structure $(M,[\nabla])$ with the projective class
represented by the connection 1-forms:
\be
\Gamma^a_{~b}=\bma \tfrac12 a\der x-\tfrac14b\der y&-\tfrac14 b\der x
& 0\\
-\tfrac14 a \der y&-\tfrac14 a\der x+\tfrac12 b\der y &0\\
c\der y-\tfrac14 a\der z&c\der x-\tfrac14 b\der z&-\tfrac14
a\der x-\tfrac14b\der y\ema 
\label{bul}\ee
The 3-manifold $M$ is parametrized by $(x,y,z)$, and $a=a(z)$,
$b=b(z)$, $c=c(z)$ are sufficiently smooth real functions of $z$.
In addition we assume that 
$$a\neq 0,\quad\quad b\neq 0,\quad\quad c\neq {\rm const}.$$
It can be checked that this connection is special. More specifically
we have:
$$W^a_{~b}=\bma
-\tfrac12 c'\der xy-\tfrac38a'\der xz+\tfrac14b'\der
yz&\tfrac38b'\der xz&\tfrac18b'\der xy\\
\tfrac38 a'\der yz&\tfrac12 c'\der xy+\tfrac14a'\der xz-\tfrac38b'\der yz&-\tfrac18a'\der
xy\\
-ac\der xy-\tfrac12c'\der yz&bc\der xy-\tfrac12c'\der
xz&\tfrac18a'\der xz+\tfrac18b'\der yz
\ema,
$$
where $(\der x\dz\der y,\der x\dz\der z,\der y\dz\der
z)=(\der xy,\der xz,\der yz)$, and
$$\om_a=\bma -\tfrac{3}{16}a^2\der x+\tfrac{1}{16}(8c'+ab)\der
y-\tfrac18a'\der z,&\tfrac{1}{16}(8c'+ab)\der x-\tfrac{3}{16}b^2\der
y-\tfrac18b'\der z,&-\tfrac18a'\der x-\tfrac18b'\der y
\ema.$$
Having these relations we easily calculate the obstructions 
$\tau_{[ed]}$. These are:
$$\tau_{13}=-\frac{9}{8192}(a')^6,\quad\quad
\tau_{23}=-\frac{9}{8192}(b')^6,$$
and $$\tau_{12}=-\tfrac{3}{128}c^2(c')^2(ba'-ab')^2.$$ 
This shows that $(M,[\nabla])$ may be metrisable only if 
$$a={\rm  const},\quad \quad b={\rm const}.$$
For such $a$ and $b$ all the obstructions $\tau_{[ed]}$ vanish. 
Assuming this we pass to the point (\ref{bui}) of our procedure from
Section \ref{bu}.

It follows that with our assumptions, the general 
solution of equation (\ref{equi}) is:
\be
g^{11}=g^{22}=0,\quad\quad g^{13}=\frac{bc}{c'}g^{12},\quad\quad g^{23}=\frac{ac}{c'}g^{12}.\label{re1}\ee
Inserting this in (\ref{ic1}), shows that its general solution is
given by the above relations for $g^{ab}$ and 
\be
\mu^1=\tfrac{1}{12}\Big(1-\frac{4cc''}{(c')^2}\Big)bg^{12},\quad\quad
\mu^2=\tfrac{1}{12}\Big(1-\frac{4cc''}{(c')^2}\Big)ag^{12}.\label{re2}\ee
The general solution (\ref{re1}), (\ref{re2}) of
(\ref{equi}), (\ref{ic1}) is compatible with
the last integrability condition (\ref{i2}) if and only if 
the function $c=c(z)$ defining our projective
structure $(M,[\nabla])$ satisfies a third order ODE:
\be
c^{(3)}c'c+\Big( (c')^2-2c c''\Big)c''=0.\label{re3}
\ee
If this condition for $c=c(z)$ is satisfied then (\ref{re1}),
(\ref{re2}) is the general solution of (\ref{equi}), (\ref{ic1}) and
(\ref{i2}). Moreover, it follows that the solution (\ref{re1}),
(\ref{re2}) also satisfies (\ref{icc}), and the tensor $g^{ab}$ is
nondegenerate for this solution provided that $g^{12}\neq 0$.  

This means that i) the projective structure $(M,[\nabla])$ with $a\neq
0$, $b\neq 0$, $c\neq {\rm const}$ may include a Levi-Civita
connection only if (\ref{re3}) holds, and ii) if it holds, that the
integrability conditions (\ref{equi}), (\ref{ic1}) and
(\ref{i2}) are all satisfied with the general solution (\ref{re1}),
(\ref{re2}), with $g^{12}\neq 0$. 

We now pass to the point (\ref{buk}) of the procedure from Section
\ref{bu}: assuming that (\ref{re3})
holds, we want to solve (\ref{sy}) for $(g^{ab},\mu^a)$ satisfying (\ref{re1}) and (\ref{re2}).

It follows that the $\{11\}$ component of the first of equations
(\ref{sy}) gives a further restriction on the function $c$. Namely, if
$(g^{ab},\mu^a)$ are as in (\ref{re1}) and (\ref{re2}), then 
$\D g^{11}=2\mu^1\theta^1$ iff $c''c-(c')^2=0$, i.e. iff 
$$c=c_1{\rm e}^{c_2z},~{\rm where}~ c_1,~
c_2~{\rm are~constants~s.t.}~ c_1c_2\neq 0.$$
Luckilly this $c$ satisfies (\ref{re3}). Looking at the next
component, $\{12\}$, 
of the first equation (\ref{sy}), we additionally get $\der g^{12}=-\tfrac12
(a\der x+b\der y)g^{12}$. And now, this is compatible with the
$\{13\}$ component of the first equation (\ref{sy}), if and only if
$b=0$ or $g^{12}=0$. We have to exclude $g^{12}=0$, since in such case
$g^{ab}$ is degenerate. On the other hand $b=0$ contradicts 
our assumptions about the function 
$b$. Thus, according to the procedure from Section \ref{bu}, we
conclude that $(M,[\nabla])$ with the connection represented by
(\ref{bul}) with $ab\neq 0$ and $c\neq {\rm const}$ \emph{never} includes a Levi-Civita connection.
\begin{remark}\label{bull}
Note that this example shows that even if all the integrability
conditions (\ref{equi}), (\ref{ic1}), (\ref{i2}) and (\ref{icc}) are
satisfied the equations (\ref{sy}) may have no solutions with
nondegenrate $g^{ab}$. Thus \emph{conditions} (\ref{equi}), (\ref{ic1}),
(\ref{i2}) \emph{and} (\ref{icc}) \emph{are not sufficient} for the existence of
a Levi-Civita connection in the projective class.     
\end{remark}
{\bf Example 2.}
As a next example we consider the same 3-dimensional manifold $M$ as
above, and equip it with a 
projective structure $[\nabla]$ corresponding to $\Gamma^a_{~b}$ as 
in (\ref{bul}), but now assuming that 
the functions $a=a(z)$ and $b=b(z)$ satisfy 
$$a\equiv 0\quad\quad{\rm and}\quad\quad b\equiv 0.$$
For the further convenience we change the variable $c=c(z)$ to the new
function $h=h(z)\neq 0$ such that $c(z)=h'(z)$.

When running through the procdure of Section \ref{bu}, which enables
us to say if such a structure includes a Levi-Civita connection,
everything goes in the same way as in the previous example, up to
equations (\ref{re2}). Thus applying our procedure of Section \ref{bu}
we get that the general solution to (\ref{equi}) and (\ref{ic1}) is
given by
$$
g^{11}=g^{22}=g^{13}=g^{23}=\mu^1=\mu^2=0.$$
It follows that this general solution to (\ref{equi}) and (\ref{ic1}),
automatically satisfies (\ref{i2}) and (\ref{icc}). 

Now, with $g^{11}=g^{22}=g^{13}=g^{23}=\mu^1=\mu^2=0$, 
the first of equations (\ref{sy}) gives:
$$g^{12}={\rm const},\quad\quad \der g^{33}=2h' g^{12}\der z,\quad\quad
\mu^3=h' g^{12},$$
and the second, in addition, gives:
$$\rho=\tfrac23h''g^{12}.$$
This makes the last of equations (\ref{sy}) automatically satisfied. 

The only differential equation to be solved is $\der g^{33}=2h'
g^{12}\der z$, which after a simple integration yields: 
$$g^{33}=2g^{12}h.$$
Thus we have 
$$g^{ab}=g^{12}\bma 0&1&0\\1&0&0\\0&0&2h\ema, $$
with the inverse
$$g_{ab}=\frac{1}{g^{12}}\bma
0&1&0\\1&0&0\\0&0&\frac{1}{2h}\ema,\quad\quad g^{12}={\rm const}\neq
0,\quad\quad h=h(z)\neq 0.$$
Now, realizing point (\ref{bum}) of the procedure of Section \ref{bu},
we define 
\be
A=-g_{ab}\mu^a\theta^b=-\frac{h'}{2h}\der z=-\tfrac12 \der\log(h).\label{bun}\ee
This means that the potential $\phi=-\tfrac12 \log(h)$, and that the
metric $\hat{g}_{ab}$ whose Levi-Civita connection is in the projective
class of 
\be
\Gamma^a_{~b}= \bma 0&0& 0\\
0&0 &0\\
h'\der y&h'\der x&0\ema,\label{buo}\ee
is given by
$$\hat{g}_{ab}=-\frac{1}{g^{12}}\bma
0&\frac{1}{h}&0\\\frac{1}{h}&0&0\\0&0&\frac{1}{2h^2}\ema,\quad\quad g^{12}={\rm const}\neq
0,\quad\quad h=h(z)\neq 0,$$
or what is the same by:
$$\hat{g}=-\frac{1}{g^{12}h^2}\big(2h\der x\der y+\der z^2),\quad\quad g^{12}={\rm const}\neq
0,\quad\quad h=h(z)\neq 0.$$
It is easy to check that in the coframe
$(\theta^1,\theta^2,\theta^3)=(\der x,\der y,\der z)$, the Levi-Civita
connection 1-forms for the metric $\hat{g}$ as above is given by
$$\hat{\Gamma}^a_{~b}=\bma
-\frac{h'}{2h}\der z&0&-\frac{h'}{2h}\der x\\
0&-\frac{h'}{2h}\der z&-\frac{h'}{2h}\der y\\
h'\der y&h'\der x&-\frac{h'}{2h}\der z
\ema,
$$
which satisfies (\ref{pt}) with $\Gamma^a_{~b}$ given by (\ref{buo})
and $A$ given by (\ref{bun}). 
\begin{remark}
Thus we have shown that the projective structure $[\nabla]$ 
generated by the
connection 1-forms (\ref{buo}) is metrisable, and that modulo
rescalling, $\hat{g}\to {\rm const}\hat{g}$, there is a \emph{unique}
metric, whose Levi-Civita connection is in the projective structure
$[\nabla]$. Note that the metric $\hat{g}$ has \emph{Lorentzian}
signature. 
\end{remark}
{\bf Example 3.} 
Now we continue with the example of a projective
structure defined in Section \ref{sec22} by formula
(\ref{exc1}). Calculating the projective Cotton tensor for this
structure we find that it is \emph{projectively flat} if and only if 
$$c''=0\quad\&\quad 2cb'+3bc'=0\quad\&\quad 2ca'+3ac'=0.$$
This happens when $a'=b'=c'=0$, but also e.g. 
when $c=z$, $b=s_1 z^{-\tfrac32}$ and $a=s_2
z^{-\tfrac32}$, with $s_1$, $s_2$ being constants. If the structure is
\emph{not} projectively flat the most general 
nondegenerate solution to equation (\ref{equi}) is
\be
g^{ab}=\bma-\frac{g^{33}}{c'} a'&g^{12}&0\\g^{12}&-\frac{g^{33}}{c'} b'&0\\0&0&g^{33}
\ema.\label{excel}\ee
It follows that if $c'=0$, projectively non flat structures which are metrisable do not
exist. In formula (\ref{excel}) we recognize (\ref{exce}) with
$f=\frac{g^{33}}{c'} $.  Looking for projectively non flat structures,
we now pass to the equation (\ref{ic1}). With $g^{ab}$ as in
(\ref{excel}) this, in particular, yields 
$$\mu^1=\mu^2=0\quad\&\quad ba'-ab'=0.$$
Thus only the structures satisfying this last equation can be
metrisable. In the following we assume that both $a$ and $b$ are
\emph{not} constant. Then $$b=s_1 a,$$ with $s_a$ a constant. 
This solution satisfies all the other equations (\ref{ic1}) if and only if 
$$\mu^3=\frac{2g^{12}(2cc'a'+a{c'}^2)+g^{33}(a'c''-c'a'')}{6a'c'}.$$
Now, with all these choices equations (\ref{i2}) are also
satisfied. Thus we may pass to the differential equations (\ref{obc1}) for the
remaining undetermined $g^{ab}$. It follows that these equations can be
satsified if and only if $$c=s_2 a$$ with $s_2={\rm const}$. Now, the
remaining equations (\ref{obc1}) are satisfied provided that the
unknown functions $g^{12}$ and $g^{33}$ satisfy:
\be
g^{12}_z=2~\tfrac{s_1}{s_2}~a~g^{33}\quad\&\quad g^{33}_z=2s_2 ~a~g^{12}\label{fni}\ee
and are independent of the variables $x$ and $y$. If $g^{12}$ and
$g^{33}$ solve (\ref{fni}) then 
all the other equations (\ref{sy}) are satisfied if and only if 
$$\rho=s_1~a^2~g^{33}+\tfrac23 ~s_2 ~a' ~g^{12}.$$
The system (\ref{fni}) can be solved explicitly (the solution is not
particularly interesting), showing that also in this
case our procedure defined in Section \ref{bu} leads effectively to
the solution of metrisability problem.


\vspace{0.3truecm}
{\bf Example 4} Our last example goes beyond 3-dimensions. It deals
with the so called (anti)deSitter spaces. 

Let $X^a$ be a \emph{constant} vector, 
and $\eta_{ab}$ be a nondegenerate symmetric $n\times n$ \emph{constant}
matrix. We focus on an example when   
$$\eta_{ab}={\rm diag}(1,\dots,1,-1,\dots,-1),$$
with $p$ `+1's, and $q$ `--1's.

In 
$${\mathcal U}=\{~(x^a)\in\bbR^n ~|~\eta_{cd}X^cx^d)\neq 0~\}$$ 
we consider metrics $\hat{g}$ of the form
\be
\hat{g}=\frac{\eta_{ab}\der x^a\der x^b}{(\eta_{cd}X^cx^d)^2}.\label{ocm}\ee
We analyse these metrics in an orthonormal coframe 
\be
\theta^a=\frac{\der x^a}{\eta_{bc}X^bx^c},\label{oc}\ee
in which 
$$\hat{g}=\eta_{ab}\theta^a\theta^b.$$
In the following we will use a convenient notation such that: 
$$\eta_{fg}X^fX^g=\eta(X,X).$$ 
We call the vector $X$ \emph{timelike} iff $\eta(X,X)>0$, \emph{spacelike} iff
$\eta(X,X)<0$, and \emph{null} iff $\eta(X,X)=0$.
 
It is an easy exercise to find that in the coframe (\ref{oc}) the Levi-Civita
connection 1-forms $\hat{\Gamma}^a_{~b}$ associated with metrics (\ref{ocm})
are:
$$\hat{\Gamma}^a_{~b}~=~\eta_{bd}~(X^a\theta^d-\X^d\theta^a).$$
Thus the Levi-Civita connection curvature, $\hat{\Omega}^a_{~b}=\der
\hat{\Gamma}^a_{~b}+\hat{\Gamma}^a_{~c}\dz\hat{\Gamma}^c_{~b}$, is
given by 
$$\hat{\Omega}^a_{~b}=-\eta(X,X)~\theta^a\dz\theta^d~\eta_{bd}.$$
This, in particular, means that the Levi-Civita curvature tensor, $\hat{R}^a_{~bcd}$, the
Levi-Civita Weyl tensor, $\weyl^a_{~bcd}$,   
and the Ricci tensor $\rlc_{ab}$, look,
respeectively, as:
$$\hat{R}^a_{~bcd}~=~\eta(X,X)~(\eta_{bc}\delta^a_{~d}-\eta_{bd}\delta^a_{~c}),$$
$$\weyl^a_{~bcd}=0,$$
and
$$\rlc_{bd}~=~(1-n)~\eta(X,X)~\eta_{bd}.$$
This proves the following proposition:
\begin{proposition}
The metrics $$\hat{g}=\frac{\eta_{ab}\der x^a\der
  x^b}{(\eta_{cd}X^cx^d)^2}$$ are the metrics of constant
curvature. Their curvature is  
totally determined by their constant Ricci scalar
$\rlc=n(1-n)\eta(X,X)$. It is positive,
vanishing or negative depending on the causal properties of the vector
$X$. Hence if $X$ is spacelike $({\mathcal U},\hat{g})$ is locally
the deSitter space, if  $X$ is timelike $({\mathcal U},\hat{g})$ is locally
the anti-deSitter space, and if $X$ is null $({\mathcal U},\hat{g})$ is flat.
\end{proposition}
Using this Proposition and Corollary \ref{wni} we see that metrics
(\ref{ocm}) are all projectively equivalent. This fact may have some
relevance in cosmology. We discuss this point in more detail in a
separate paper \cite{nur}.

\end{document}